\documentclass[10pt,reqno]{amsart}

\usepackage{amssymb}
\usepackage{amsfonts}
\usepackage{amsthm}
\usepackage{amsmath}
\usepackage{bbm}

\numberwithin{equation}{section}
\allowdisplaybreaks

\newtheorem{theorem}{Theorem}[section]
\newtheorem{proposition}[theorem]{Proposition}

\newtheorem{lemma}[theorem]{Lemma}
\newtheorem{corollary}[theorem]{Corollary}

\theoremstyle{definition}

\theoremstyle{remark}

\newcommand{\R}{\mathbb{R}}

\newcommand{\Z}{\mathbb{Z}}

\renewcommand{\hat}{\widehat}
\newcommand{\eps}{\varepsilon}

\newcommand{\scriptE}{\mathcal{E}}
\newcommand{\scriptF}{\mathcal{F}}

\newcommand{\qtq}[1]{\quad\text{#1}\quad}

\DeclareMathOperator*{\supp}{supp}

\DeclareMathOperator*{\dist}{dist}

\begin{document}
\title[Fourier restriction to polynomial curves]{Uniform estimates for Fourier restriction to polynomial curves in $\R^d$}
\author{Betsy Stovall}
\address{Department of Mathematics, University of Wisconsin, Madison, WI 53706}
\email{stovall@math.wisc.edu}
\begin{abstract}
We prove uniform $L^p \to L^q$ bounds for Fourier restriction to polynomial curves in $\R^d$ with affine arclength measure, in the conjectured range.  
\end{abstract}

\maketitle

\section{Introduction} 

In this article, we consider the problem of restricting the Fourier transform of an $L^p$ function on $\R^d$ to a  curve $\gamma:\R \to \R^d$.  It has been known for a number of years that when the curve is equipped with Euclidean arclength measure, the $p$ and $q$ such that this defines a bounded operator from $L^p$ to $L^q$ must depend on the maximal order of vanishing of the torsion $L_\gamma:=\det(\gamma',\ldots,\gamma^{(d)})$.  However, when the curve is equipped with affine arclength measure, $\lambda_\gamma\, dt = |L_\gamma|^{\frac 2{d(d+1)}}\, dt$, any degeneracies of curvature are mitigated and the known torsion-dependent obstructions vanish.  

Because the sharp $L^p(dx) \to L^q(\lambda_\gamma\, dt)$ estimates for the restriction operator are completely invariant under affine transformations of $\R^d$ and reparametrizations of $\gamma$, there has been considerable interest (such as \cite{BOS09, BOS08, BOS11, DeMu, DW, DrMa85, DrMa87, Sjolin}) in the question of whether such bounds hold uniformly over certain large classes of curves.  This is part of a broader program (\cite{CKZ1, CZ, DSjfa,  OberlinPAMS12, OberlinCamb00} and many others) to determine whether curvature-dependent bounds for various operators arising in harmonic analysis can be generalized, uniformly, by the addition of appropriate affine arclength or surface measures.  

In the case of Fourier restriction to curves, the only known obstruction to $L^p(dx)\to L^q(\lambda_\gamma\, dt)$ boundedness is oscillation.  The example $(t,e^{-1/t}\sin(t^{-k}))$, $0<t<1$ is due to Sj\"olin (\cite{Sjolin}); an example with nonvanishing torsion is a line of irrational slope on the two-torus in $\R^3$.  Motivated by this, a natural question, suggested by Dendrinos--Wright in \cite{DW}, is whether there hold restriction estimates with constants that are uniform over the class of polynomial curves of any fixed degree.  This question has been settled in the affirmative in dimension two \cite{Sjolin}.  In higher dimensions, the result has been proved for general polynomial curves in a restricted range \cite{BOS11, DW}, and for monomial \cite{BOS09, Drury85} and `simple' \cite{BOS11} polynomial curves in the full range.  Our main theorem settles this question in the remaining conjectured cases.  

\begin{theorem}\label{T:main}  For each $N$, $d$, and $(p,q)$ satisfying 
\begin{equation} \label{E:restriction pq}
p'=\tfrac{d(d+1)}2q, \qquad q > \tfrac{d^2+d+2}{d^2+d}, 
\end{equation}
there exists a constant $C_{N,d,p}$ such that for all polynomials $\gamma:\R \to \R^d$ of degree less than or equal to $N$,
\begin{equation} \label{E:restriction}
\|\hat f(\gamma(t)) \|_{L^q(\lambda_\gamma\, dt)} \leq C_{N,d,p}\|f\|_{L^q(dx)},
\end{equation} 
for all Schwartz functions $f$.  
\end{theorem}

This is sharp.  A simple scaling argument shows that \eqref{E:restriction} can only (provided $\lambda_\gamma \not\equiv 0$) hold if $p' = \tfrac{d(d+1)}2 q$, and Arkhipov--Chubarikov--Kuratsuba proved in \cite{AKC} that the restriction $q > \tfrac{d^2+d+2}{d^2+2d}$ is also necessary.  At the endpoint $q_d = \tfrac{d^2+d+2}{d^2+d}$ (at which point $p=q$), there are some cases where the corresponding restricted strong type bound is known \cite{BOS09, BOS11}; we do not address the endpoint case here.  

The parametrization and affine invariance of \eqref{E:restriction} make the affine arclength measure a natural object of study.  This aptness is underscored by the facts that it is essentially the largest positive measure such that the above $L^p \to L^q$ bounds can hold (\cite{Oberlin03}, which also has an interesting geometric perspective), and in the case of a compact nondegenerate curve, the interpolation of the above estimates with elementary ones can be used to deduce essentially all valid $L^p \to L^q$ inequalities \cite{AKC}.  In the general polynomial case, we will show that Theorem~\ref{T:main} and an interpolation argument imply an extension into Lorentz spaces, as well as the full range of estimates for the unweighted operator.  Let $K_{\rm{min}}$ equal the maximum order of vanishing of $L_\gamma$ on $\R$ and let $K_{\rm{max}}$ equal the degree of $L_\gamma$.  Let $N_{\bullet} = K_{\bullet}+\tfrac{d^2+d}2$.  

\begin{corollary}\label{C:interp}  
Let $d \geq 3$ and let $\gamma:\R \to \R^d$ be a polynomial curve, with $L_\gamma \not\equiv 0$.  Define $d_\gamma(t) = \dist(t,Z_\gamma)$, where $Z_\gamma$ is any finite set containing the complex zeros of $L_\gamma$.  Then for all $1 < p < \tfrac{d^2+d+2}{d^2+d}$ and $q \leq \tfrac{2p'}{d(d+1)}$, 
\begin{equation} \label{E:DM interp}
\|\hat f(\gamma(t)) |L_\gamma(t)|^{\frac1{p'}} d_\gamma(t)^{-\frac1q+\frac{d(d+1)}{2p'}}\|_{L^{q,p}(dt)} \leq C_{p,q,\deg \gamma,\#Z_\gamma} \|f\|_{L^p(dx)}.
\end{equation}
Furthermore, the unweighted operator satisfies
\begin{equation} \label{E:unweighted}
\|\hat f(\gamma(t))\|_{L^q(dt)} \leq C_{p,q,\gamma}\|f\|_{L^p(dx)},
\end{equation}
if and only if $1 \leq p < \tfrac{d^2+d+2}{d^2+d}$, and either $p \leq q$ and $N_{\rm{min}} q \leq p' \leq N_{\rm{max}}q$, or $p > q$ and $N_{\rm{min}}q < p' < N_{\rm{max}}q$.  
\end{corollary}

The use of the weight $d_\gamma$ and the resulting uniformity in \eqref{E:DM interp} seem to be new, but the estimate and its proof are inspired by an argument in \cite{DrMa85}; \eqref{E:unweighted}  sharpens and makes global certain estimates appearing in \cite{ChTAMS, DrMa85, DrMa87}.  These bounds are valid even for $q<1$ in the given range.  The dependence of the constant in \eqref{E:unweighted} on $\gamma$ is unavoidable because of the lack of affine and parametrization invariance.      Related estimates in the context of generalized Radon transforms will also appear in \cite{DSinprep}.  We note that the techniques used in our proof could also be used to obtain the full range of estimates (necessarily nonuniform) for restriction with Euclidean arclength measure, but the exponents would be a bit more complicated.  

\subsection*{Prior results}

The problem of obtaining uniform bounds for restriction with affine arclength dates back to the 70s, when Sj\"olin \cite{Sjolin} proved a sharp restriction result that is completely uniform over the class of convex plane curves.  This result implies the two dimensional version of Theorem~\ref{T:main} by the triangle inequality.  

In higher dimensions, the first results are due to Prestini in \cite{Prestini3, Prestini_n}, who proved restriction estimates for nondegenerate curves in a restricted range, just off the sharp line.  In \cite{ChTAMS}, Christ extended Prestini's result to the sharp line and the range $q \geq \tfrac{d^2+2d}{d^2+2d-2}$, and obtained new estimates for unweighted restriction to certain degenerate curves.  Shortly thereafter, Drury \cite{Drury85} extended Christ's result to the full range ($q > \tfrac{d^2+d+2}{d^2+d}$) for nondegenerate curves.  

For degenerate curves, the development was somewhat slower.  In \cite{DrMa85, DrMa87}, Drury--Marshall ultimately established sharp, global bounds in the so-called Christ range and local estimates in the interior of the conjectured region for restriction to monomial curves (with arbitrary real powers).  It was not for another twenty years, when Bak--Oberlin--Seeger \cite{BOS09} proved \eqref{E:restriction} for monomial curves in the range \eqref{E:restriction pq}, that any sharp estimates were known beyond the Christ range for any flat curves; \cite{BOS09} also established the endpoint (restricted strong type) estimate for nondegenerate curves.  Very shortly thereafter, Dendrinos--Wright \cite{DW} posed the problem considered here and also proved Theorem~\ref{T:main} in the Christ range.  In \cite{DeMu}, Dendrinos--M\"uller proved that the estimates for monomial curves from \cite{BOS09} are stable under sufficiently small perturbations.  Shortly thereafter, Bak--Oberlin--Seeger \cite{BOS11} established uniform bounds in the full range \eqref{E:restriction pq} for `simple' polynomial curves, i.e.\ curves of the form $(t,t^2,\ldots,t^{d-1},P(t))$, and slightly extended the Dendrinos--Wright range for general polynomial curves.  Our result is new in the remaining cases.

\subsection*{Outline of proof}  Though we build on much of the above-mentioned literature (especially \cite{ChTAMS, DW, Drury85, DrMa85}), we take a new approach for the degenerate case, particularly compared to the recent \cite{BOS09,BOS11,DeMu}, by using a dyadic decomposition according to torsion size, coupled with a square function estimate.

We begin in Section~\ref{S:local rest} by considering the problem of restricting $\hat f$ to a segment along $\gamma$ that has roughly constant torsion.  By a scaling and compactness argument, together with an induction argument from \cite{Drury85}, we establish uniform estimates along this segment without facing the significantly more delicate task of performing real interpolation in the presence of the affine arclength (cf.\ \cite{BOS09, DeMu}).  We close with a more detailed discussion of the techniques in other recent articles.

In order to use the estimates from Section~\ref{S:local rest}, we must decompose the operator according to the size of the torsion, while not ruining our chances of being able to put the pieces back together later.  This we do in Section~\ref{S:square} by means of a uniform square function estimate for the extension operator.  The heuristic behind this is the (false) assertion that if the torsions of two points on the curve are at different scales, then the points themselves must be at different frequency scales.

We complete the proof of Theorem~\ref{T:main} in Section~\ref{S:sum}.  Using an argument inspired by the recent success of bilinear and multilinear approaches to Fourier restriction to hypersurfaces, we reduce matters to proving a $\tfrac{d(d+1)}2$-linear extension estimate, which exhibits decay when the arguments live at different torsion scales on the curve.  The core of the argument is a variant of Christ's multilinear estimate, together with interpolation with our bounds from Section~\ref{S:local rest}.  

In Section~\ref{S:interp}, we prove the corollary.  

It would be interesting to see whether these ideas could be used to obtain more uniform bounds for restriction to sufficiently smooth finite type curves, or whether some of the simplifications here could lead to progress on the endpoint restricted strong type bounds in the general polynomial case.  Somewhat more broadly, our approach of transferring estimates from the non-degenerate to the degenerate case, particularly the use of the square function estimate and multilinear estimates with decay, may be useful in establishing bounds for Fourier restriction to degenerate submanifolds of higher dimension.  

\subsection*{Acknowledgements}  The author would like to thank Andreas Seeger and Spyros Dendrinos for enlightening discussions about the history of this problem (though certainly any omissions or errors in this regard are the author's).  She would also like to thank the anonymous referee.  This work was supported by NSF grant DMS-1266336.  

\subsection*{The dual formulation and other notation}

In proving Theorem~\ref{T:main}, we will focus on the corresponding extension problem.  Fix a polynomial $\gamma:\R \to \R^d$.  Define the weighted and unweighted extension operators
\begin{gather*}
\scriptE_\gamma f(x) = \int_\R e^{ix\gamma(t)}f(t)\, \lambda_\gamma(t)\, dt \qquad\qquad
\scriptF_\gamma f(x) = \int_\R e^{ix\gamma(t)}f(t)\, dt.
\end{gather*}
Theorem~\ref{T:main} is equivalent to showing that 
\begin{equation} \label{E:lplq}
\|\scriptE_\gamma f\|_{L^q(dx)} \leq C_{d,p,N}\|f\|_{L^p(\lambda_\gamma\, dt)},
\end{equation}
for all $(p,q)$ in the range 
\begin{equation} \label{E:admissibility}
q=\tfrac{d(d+1)}2 p', \qquad q > \tfrac{d^2+d+2}2.
\end{equation}

An affine transformation is a map of the form $Ax = Mx+x_0$, where $M$ is a $d \times d$ matrix and $x \in \R^d$.  We will denote by $\det A$ the determinant of the corresponding linear map, $\det A = \det M$.  

As usual, if $B_1$ and $B_2$ are two non-negative quantities, we will write $B_1 \lesssim B_2$ if $B_1 \leq CB_2$ for some innocuous constant $C$.  These constants will be allowed to change from line to line.  

\section{Uniform local restriction} \label{S:local rest}

The bulk of this section will be devoted to a proof of the following theorem.  We will conclude the section with a detailed comparison of our approach with the recent literature.

\begin{theorem}\label{T:local rest}  Fix $d \geq 2$, $N$, and $(p,q)$ satisfying \eqref{E:admissibility}.  For every interval $I \subset \R^d$ and every degree $N$ polynomial $\gamma:\R \to \R^d$ satisfying
\begin{equation} \label{E:roughly constant}
0 < C_1 \leq |\det(\gamma'(t),\ldots,\gamma^{(d)}(t))| \leq C_2, \qquad t \in I,
\end{equation}
we have the extension estimate
\begin{equation} \label{E:local extn}
\|\mathcal E_\gamma (\chi_I f)\|_{L^q} \leq C_{d,N,p,\frac{C_2}{C_1}}\|f\|_{L^p(\lambda_\gamma)}.
\end{equation}
\end{theorem}

We will use the notation:
\begin{gather*}
L_\gamma(t) = \det(\gamma'(t),\ldots,\gamma^{(d)}(t)), \qquad J_\gamma(t_1,\ldots,t_d) = \det(\gamma'(t_1),\ldots,\gamma'(t_d)), \\
v(t_1,\ldots,t_d) = c_d\prod_{1 \leq i < j \leq d}(t_j-t_i) \:\:\text{is the Vandermonde determinant.}
\end{gather*}

Since $L_\gamma$ is a polynomial of degree less than $Nd$, we may write $I = \bigcup_{j=1}^{CNd \log(\frac{C_2}{C_1})} I_j$ where 
$$
\tfrac12 C \leq |L_\gamma(t)| \leq 2C, \qquad t \in I_j.
$$
By the triangle inequality, it suffices to prove \eqref{E:local extn} with $I$ replaced by one of the $I_j$.  Utilizing the affine and parametrization invariances, we may assume that $C=1$ and $I=[-1,1]$.  In other words, we may assume that $\gamma$ satisfies 
\begin{equation} \label{E:on I}
\tfrac12 \leq |L_\gamma(t)| \leq 2, \qquad t \in [-1,1],
\end{equation}
and we want to prove
\begin{equation} \label{E:local extn'}
\|\mathcal E_\gamma(\chi_I f)\|_{L^q} \leq C_{d,N,p} \|f\|_{L^p(\lambda_\gamma)}, 
\end{equation}
where $I=[-1,1]$ and $(p,q)$ satisfies \eqref{E:admissibility}.  

\begin{lemma} \label{L:CN}
If $\gamma:\R \to \R^d$ is a degree $N$ polynomial satisfying \eqref{E:on I}, there exists an affine transformation $A$ with $\det A = 1$ and $\|A \gamma\|_{C^N([-1,1])} \leq C_{N,d}$.  
\end{lemma}  

\begin{proof}  Let $A$ denote the affine transformation 
$$
Ax = [\begin{array}{lcr} \gamma'(0) & \cdots & \gamma^{(d)}(0)\end{array}]^{-1} (x-\gamma(0)).
$$
Then $\tfrac12 \leq \det A \leq 2$ by \eqref{E:on I}, so (replacing $A$ with $(\det A^{-1})A$ if needed) the lemma will be proved if we can show that $\|A \gamma\|_{C^N([-1,1])} \leq C_{N,d}$.  

To simplify the notation, we assume henceforth that $A\gamma = \gamma$.  Thus
\begin{equation} \label{E:torsion bound}
\tfrac14 \leq L_\gamma(t) \leq 4, \quad t \in [-1,1], \qquad \gamma(0) = 0, \qquad \gamma^{(j)}(0) = e_j, \quad 1 \leq j \leq d.
\end{equation}
By Taylor's theorem, it suffices to show that $|\gamma^{(j)}(0)| \leq C_{d,N}$ for all $j \geq 1$, and hence by \eqref{E:torsion bound}, it suffices to show that 
\begin{equation} \label{E:torsion bounds higher}
|\det(\gamma^{(n_1)}(0),\ldots,\gamma^{(n_d)}(0))| \leq C_{d,N}, 
\end{equation}
for all $n_1 < \cdots < n_d$; this is because $[\gamma'(0),\ldots,\gamma^{(d)}(0)]$ is the identity, so we can replace any column by $\gamma^{(n)}(0)$ to pick out the coefficient we want.

Suppose that the lemma is false.  Then there exists a sequence of polynomials $\gamma_n$ satisfying \eqref{E:torsion bound} such that
\begin{equation} \label{E:to infty}
\max_{1 \leq n_1 < \cdots < n_d \leq N} |\det(\gamma_n^{(n_1)}(0),\ldots,\gamma_n^{(n_d)}(0))| \to \infty.  
\end{equation}

Let $(\delta_n)$, $0 < \delta_n< 1$, be a sequence, to be determined in a moment.  Define rescaled curves by 
$$
\Gamma_n(t) = (\delta_n^{-1}\gamma_{n,1}(\delta_n t),\ldots,\delta_n^{-d}\gamma_{n,d}(\delta_n t)).
$$
Observe that $\Gamma_n$ obeys \eqref{E:torsion bound} on $[-\delta_n^{-1},\delta_n^{-1}]$ and that if $n_1 + \cdots + n_d > \frac{d(d+1)}2$, 
$$
|\det(\Gamma_n^{(n_1)}(t),\ldots,\Gamma_n^{(n_d)}(t))| \leq \delta_n |\det(\gamma_n^{(n_1)}(\delta_n t),\ldots,\gamma_n^{(n_d)}(\delta_n t))|.
$$
By \eqref{E:to infty}, for each $n$ sufficiently large, we may choose $0 < \delta_n < 1$ so that
\begin{equation} \label{E:bound Gamma}
\max_{1 \leq n_1 < \cdots < n_d \leq N} |\det(\Gamma_n^{(n_1)}(0),\ldots,\Gamma_n^{(n_d)}(0))| = 5;
\end{equation}
\eqref{E:to infty} further implies that $\delta_n \to 0$.  Passing to a subsequence, there exists a single $d$-tuple $1 \leq n_1 < \cdots < n_d \leq N$ such that 
\begin{equation} \label{E:equals 5}
|\det(\Gamma_n^{(n_1)}(0),\ldots,\Gamma_n^{(n_d)}(0))| = 5, \quad \text{for all $n$.}
\end{equation}
By \eqref{E:torsion bound}, $n_1+\cdots+n_d > \tfrac{d(d+1)}2$.

On the other hand, by \eqref{E:bound Gamma}, \eqref{E:torsion bound}, and the observation after \eqref{E:torsion bounds higher}, $|\Gamma_n^{(j)}(0)| \leq C_{N,d}$ for all $j$.  In other words, all of the coefficients of $\Gamma_n$ are bounded.  Thus after passing to a subsequence, there exists a limit, $\Gamma_n \to \Gamma$ (in the metric space of polynomial curves of degree at most $N$).  By \eqref{E:torsion bound} and the fact that $\delta_n \to 0$, 
$$
\tfrac14  \leq |L_\Gamma(t)| \leq 4, \qquad t \in \R.
$$
So by \eqref{E:torsion bound}, $L_\Gamma(t) \equiv 1$, since $\Gamma$ is a polynomial.  This implies that $$
\Gamma(t) = (t,\tfrac12t^2,\ldots,\tfrac1{d!}t^d)
$$
 (no affine transformation is necessary by \eqref{E:torsion bound}).  But by \eqref{E:equals 5},
$$
|\det(\Gamma^{(n_1)}(0),\ldots,\Gamma^{(n_d)}(0))| = 5,
$$
a contradiction.  This completes the proof.  
\end{proof}

Theorem~\ref{T:local rest} almost follows from Lemma~\ref{L:CN} by a result of Drury (Theorem~2 of  \cite{Drury85}), but some additional uniformity is needed.  The first step is to obtain estimates for the offspring curves of $\gamma$.

\begin{lemma}\label{L:uniform DM}  Fix $d \geq 2$ and $N$.  There exists a constant $c_d > 0$ and a decomposition $[-1,1] = \bigcup_{j=1}^{M_{d,N}} I_j$ into disjoint intervals such that the conclusions below hold for every $I=I_j$ and every degree $N$ polynomial $\gamma:\R \to \R^d$ satisfying
\begin{equation}\label{E:Lgamma uniform DM}
\tfrac12 \leq |L_\gamma(t)| \leq 2, \qquad t \in [-1,1].
\end{equation}
If $K\geq 1$ and $(h_1,\ldots,h_K) \in \R^K$, the curve defined by $\gamma_h(t) = \tfrac1K\sum_{j=1}^K \gamma(t+h_j)$ satisfies the following on the interval $I_h : = \cap_{j=1}^K (I-h_j)$:\\
\begin{gather}
\label{E:geometric ineq}
|J_{\gamma_h}(t_1,\ldots,t_d)| \geq c_d \prod_{j=1}^d |L_{\gamma_h}(t_j)|^{\frac 1d} \prod_{1 \leq i < j \leq d}|t_j-t_i|\\
\label{E:bigger torsion}
c_d  \leq |L_{\gamma_h}(t)| \leq c_d^{-1}
\end{gather}
\end{lemma}

In \cite{Drury85}, an argument is given to prove an analogous lemma with the weaker hypothesis that $\gamma$ is $C^d$.  In that more general case, the uniformity we need is impossible, so we give a detailed proof of Lemma~\ref{L:uniform DM} (via a different argument) here.  

\begin{proof}  Fix a curve $\gamma$ satisfying the hypotheses of the lemma and let $t_0 \in [-1,1]$.  Let $\delta>0$ be a sufficiently small constant, depending only on $d$ and $N$ and to be determined in a moment.  We will show that the conclusions of the lemma hold on $I := [t_0-\delta,t_0+\delta]$.  This is sufficient.

By Lemma~\ref{L:CN}, by reparametrizing and performing an affine transformation on $\gamma$ if necessary, we may assume that $t_0 = 0$, $\gamma(0) = 0$, $\gamma^{(j)}(0) = e_j$, $1 \leq j \leq d$, and $|\gamma^{(j)}(0)| \leq C_{N,d}$ for all $j \geq 1$.  Therefore
$$
\gamma(t) = (t,\tfrac12 t^2, \ldots,\tfrac1{d!}t^d) + \tilde\gamma(t),
$$
where $\tilde \gamma$ is a degree $N$ polynomial with $\tilde\gamma^{(j)}(0) = 0$ for $0 \leq j \leq d$ and $|\tilde\gamma^{(j)}(0)| \leq C_{N,d}$ for all $j$.  

Now fix $K \geq 1$ and $h \in \R^K$.  We can translate $h$, which just shifts $I_h$, and reorder its components with impunity, so without loss of generality, $0=h_1 < h_2 < \cdots < h_K$.  If $I_h = \emptyset$, the conclusions of the lemma are trivial, so we may assume that $h_K < 2\delta$.  

Define a matrix $A_h$ by 
$$
(A_h)_{ij} = 
\begin{cases} 
0, \quad &\text{if $i<j$};\\
\tfrac1K\sum_{k=1}^K\tfrac{h_k^{i-j}}{(i-j)!}, \quad &\text{if $i\geq j$}.
\end{cases}
$$
In particular, $A_h$ is lower triangular with ones on the diagonal, so it is invertible.  Define
$$
\Gamma_h(t) =  A_h^{-1}(\gamma_h(t) - \tfrac1K\sum_{j=1}^K(h_j,\tfrac12h_j^2,\ldots,\tfrac1{d!}h_j^d)), \qquad \tilde\Gamma_h(t) = A_h^{-1}\tilde\gamma_h(t).
$$
Since
$$
A_h (t,\ldots,\tfrac1{d!}t^d) = \tfrac1K\sum_{k=1}^K (t+h_j,\ldots,\tfrac1{d!}(t+h_j)^d) - \tfrac1K\sum_{k=1}^K(h_j,\ldots,\tfrac1{d!}h_j^d),
$$
we have
$$
\Gamma_h(t) = (t,\ldots,\tfrac1{d!}t^d) + \tilde\Gamma_h(t).
$$

Since $\|A_h-I\| \lesssim_d \delta$, for $\delta$ sufficiently small (depending only on $d$), $\|A_h^{-1}-I\| \lesssim_d \delta$.  Furthermore, because $\tilde\gamma^{(j)}(0) = 0$ for $1 \leq j \leq d$ and $|\tilde\gamma^{(j)}(0)| \leq C_{d,N}$ for all $j$, $|\tilde\gamma^{(j)}_h(0)| \leq C_{d,N}\delta^{d+1-j}$ for $1 \leq j \leq d$ and $|\tilde\gamma^{(j)}_h(0)| \leq C_{d,N}$ for all $j$.  Combining these bounds
$$
|\tilde\Gamma_h^{(j)}(t)| \leq C_{d,N}\delta^{d+1-j}, \quad  |\tilde \Gamma_h^{(k)}(t)| \leq C_{d,N}, \qtq{for all} \: 1 \leq j \leq d \leq k, \quad  t \in [-\delta,\delta].
$$

From this and multilinearity of the determinant,
$$
L_{\gamma_h}(t) = L_{\Gamma_h}(t) = 1 + O_{d,N}(\delta),
$$
which implies \eqref{E:bigger torsion} for $\delta$ sufficiently small.

Since $J_{\Gamma_h}$ is an antisymmetric polynomial, for any $s \in \R$, we can write
\begin{align} \label{E:factor J}
&J_{\Gamma_h}(t_1,\ldots,t_d) = P_{\Gamma_h}(t_1,\ldots,t_d)\prod_{1 \leq i < j \leq d}(t_j-t_i)\\\notag
&\qquad= c_d \sum_{\sigma\in S_d} \rm{sgn}(\sigma) P_{\Gamma_h}(t_1,\ldots,t_d)(t_{\sigma(2)}-s) (t_{\sigma(3)}-s)^2 \cdots (t_{\sigma(d)}-s)^{d-1},
\end{align}
where $P_{\Gamma_h}$ is a symmetric polynomial of degree less than $Nd$ and $S_d$ denotes the symmetric group on $d$ letters.  Using the second line of \eqref{E:factor J}, we differentiate $J_{\Gamma_h}$ term-by-term and evaluate at $(s,\ldots,s)$:
$$
P_{\Gamma_h}(s,\ldots,s) = b_d \partial_d^{d-1}\partial_{d-1}^{d-2} \cdots \partial_2|_{t=(s,\ldots,s)}J_{\Gamma_h}(t);
$$
this is because the derivatives must fall on the Vandermonde term.  (Here, $b_d$ is a dimensional constant which will be allowed to change from line to line.)  On the other hand, 
$$
\partial_d^{d-1} \cdots \partial_2|_{t=(s,\ldots,s)} J_{\Gamma_h}(t) = L_{\Gamma_h}(s),
$$
so 
$$
P_{\Gamma_h}(s,\ldots,s) = b_d L_{\Gamma_h}(s).
$$

In addition, because $|\Gamma_h^{(j)}(s)| \leq C_{N,d}$ for all $s \in [-\delta,\delta]$ and all $j \geq 1$, the derivatives of $P_{\Gamma_h}$ also satisfy
$$
|\partial^\alpha P_{\Gamma_h}(s,\ldots,s)| \leq C_{d,N},
$$
for all multiindices $\alpha$ and all $s \in [-\delta,\delta]$.  Therefore,
$$
P_{\Gamma_h}(t) = P_{\Gamma_h}(s,\ldots,s) + O_{d,N}(\delta) = b_d L_{\Gamma_h}(s)+O_{d,N}(\delta) = b_d+O_{d,N}(\delta),
$$
for all $(t_1,\ldots,t_d) \in [-\delta,\delta]^d$ and $s \in [-\delta,\delta]$.  Combining this with \eqref{E:factor J}, we obtain \eqref{E:geometric ineq}.  

This completes the proof of the proposition.
\end{proof}

Drury's induction argument from \cite{Drury85} completes the proof of the theorem.  (A previous draft had used the method of \cite{BOS09, DeMu}; Spyros Dendrinos kindly pointed out that the estimates in Lemmas~\ref{L:CN} and~\ref{L:uniform DM} meant that Drury's approach, which is somewhat more direct, could be used.)  The base case is the trivial observation that
\begin{equation} \label{E:L1Linfty}
\|\scriptE_\gamma f\|_{L^\infty} \leq \|f\|_{L^1(\lambda_\gamma)},
\end{equation}
for any function $f$ and any $C^d$ curve $\gamma$.  Our hypothesis is the statement that for some $1 \leq p < \tfrac{d^2+d+2}2$, there exists a constant $C_{d,p}$ such that
\begin{equation}\label{E:hyp}
\|\scriptE_{\gamma_h}(\chi_{I_h}f)\|_{L^{\frac{d(d+1)p'}2}} \leq C_{d,p}\|f \|_{L^p(\lambda_\gamma)}, \:\: \text{for all $K \geq 1$, $h \in \R^K$}.
\end{equation}
We want to increase $p$ (decreasing $p$ is easy by interpolation with \eqref{E:L1Linfty}).  The inductive step is the following.  

\begin{lemma}[\cite{Drury85}] \label{L:Drury}
If the hypothesis \eqref{E:hyp} is valid for some $1 \leq p_0 < \tfrac{d^2+d+2}2$, then it is also valid for all $p \geq 1$ satisfying the inequality
$$
\tfrac dp > \tfrac2{d+2} + \tfrac{(d-2)p_0^{-1}}{d+2}.
$$
\end{lemma}

For the convenience of the reader and to make it easier to describe the background in more detail, we sketch the argument below.  Complete details are given in \cite{Drury85} in the case $\gamma = (t,t^2,t^3)$.  We recall the notation
$$
v(h) = c_d\prod_{1 \leq i < j \leq d} (h_j-h_i).
$$

\begin{proof}[Sketch of proof]
Let $\tilde\gamma$ be an offspring curve: $\tilde\gamma = \gamma_h: \tilde I = I_h \to \R^d$ for some $K \geq 1$, $h \in \R^K$.  By virtue of \eqref{E:bigger torsion}, it suffices to bound the unweighted operator
$$
\scriptF_{\tilde\gamma} f(x) = \int_{\tilde I} e^{ix\tilde \gamma(t)}f(t)\, dt.
$$
We denote by $\mu$ the measure defined by
$$
\mu(\phi) = \int_{\tilde I} \phi(\tilde\gamma(t))\, f(t)\, dt.
$$
Since 
\begin{equation} \label{E:*^d}
\|\scriptF_{\tilde\gamma} f\|_{L^p} = \|\check\mu^d\|_{L^{\frac pd}}^{\frac1d},
\end{equation}
we are interested in the $d$-fold convolution of $\mu$.

If $g(\xi) = \mu*\cdots*\mu(d\xi)$, a computation shows that
$$
g(\tfrac1d(\tilde\gamma(t_1) + \cdots + \tilde\gamma(t_d))) = \tfrac{c_d}{J_{\tilde\gamma}(t_1,\ldots,t_d)}f(t_1) \cdots f(t_d).
$$
(We must change variables to do this.  It is a consequence of \eqref{E:geometric ineq} and \eqref{E:bigger torsion} that $(t_1,\ldots,t_d) \mapsto \sum \tilde\gamma(t_j)$ is one-to-one on $\{t \in \tilde I^d : t_1 < \cdots < t_d\}$.  See e.g. \cite[Section~3]{DrMa87}.  Alternatively, $C_{d,N}$-to-one follows from Bezout's theorem.)  

For $h=(h_1,h') = (0,h') \in \{0\}\times\R^{d-1}$, we define
$$
G(t;h) = g(\tfrac1d(\tilde\gamma(t+h_1) + \cdots + \tilde\gamma(t+h_d))).
$$
By a change of variables in the $\xi$ variable,
$$
\hat g(x) = C_d \int_{\R^{d-1}}\int_{\tilde I_h} e^{ix\tilde\gamma_h(t)}g(\tilde\gamma_h(t))\, |J_{\tilde\gamma}(t+h_1,\ldots,t+h_d)|^{-1}\, dt\, dh',
$$
where here and for the remainder of the section, we use the convention that $h_1=0$.  By Plancherel, \eqref{E:geometric ineq}, and \eqref{E:bigger torsion},
$$
\|\hat g\|_{L^2_x} \leq C_d \|G\|_{L^2_{h'}(L^2_t; |v(h)|^{-1})}, 
$$
and by \eqref{E:hyp} (plus interpolation to decrease $p$) and the integral form of Minkowski's inequality,
$$
\|\hat g\|_{L^q_x} \leq C_{d,p} \|G\|_{L^1_{h'}(L^p_t; |v(h)|^{-1})}, \quad 1 \leq p \leq p_0, \:\: q=\tfrac{d(d+1)}2 p_0'.
$$

Thus by interpolation, 
\begin{equation} \label{E:hat g}
\|\hat g\|_{L^c_x} \leq C_{d,a,b} \|G\|_{L^a_{h'}(L^b_t; |v(h)|^{-1})}, 
\end{equation}
for all $(a^{-1},b^{-1})$ in the triangle with vertices $(1,1), (1,p_0^{-1}), (\frac12,\frac12)$ and $c$ satisfying
$$
\tfrac{(d+2)(d-1)}2 a^{-1} + b^{-1}  + \tfrac{d(d+1)}2 c^{-1} = \tfrac{d(d+1)}2.
$$
A computation (again using \eqref{E:geometric ineq} and \eqref{E:bigger torsion}) shows that
$$
\|G\|_{L^a_{h'}(L^b_t; |v(h)|^{-1})} \sim_{d,a,b} \bigl\{ \int |v(h)|^{-(a-1)}(\int_{\tilde I_h} |f(t+h_1) \cdots f(t+h_d)|^b\, dt)^{\frac ab}\, dh'\bigr\}^{\frac1a}.
$$
Using this and the fact that $v(0,h') \in L^{\frac d2,\infty}_{h'}$, one can show that 
\begin{gather} \label{E:Gab}
\|G\|_{L^a_{h'}(L^b_t; |v(h)|^{-1})} \leq \|f\|_{L^{p,1}_t}^d, \qquad \text{for}\\\notag
1 < a < \tfrac{d+2}d, \qquad a \leq b < \tfrac{2a}{d+2-da}, \qquad \tfrac dp = \tfrac{(d+2)(d-1)}2 a^{-1} + b^{-1} - \tfrac{d(d-1)}2.
\end{gather}
Details may be found in \cite{BOS09, Drury85}.  

By \eqref{E:*^d}, \eqref{E:hat g}, \eqref{E:Gab}, and some careful arithmetic,
\begin{equation} \label{E:real interp}
\|\scriptF_{\tilde\gamma} f\|_{L^q} \leq C_{d,q} \|f\|_{L^{p,1}},
\end{equation}
for $p,q,a,b$ satisfying 
\begin{gather*}
q = \tfrac{d(d+1)}2 p', \qquad \tfrac dp = \tfrac{(d+2)(d-1)}2 a^{-1}+b^{-1} - \tfrac{d(d-1)}2, \\
 \tfrac d{d+2} < a^{-1} < 1, b^{-1} \leq a^{-1}, \qquad \tfrac{d+2}a-\tfrac2b < d, \qquad (p_0-2)a^{-1}+ p_0 b^{-1} \geq p_0-1.  
\end{gather*}

The point $(a^{-1},b^{-1}) = (\tfrac d{d+2},\tfrac 2{d+2}+\tfrac{d-2}{(d+2)p_0})$ lies on the boundary of this region and satisfies
\begin{equation} \label{E:ab good}
\tfrac{(d+2)(d-1)}2 a^{-1} + b^{-1} - \tfrac{d(d-1)}2 < \tfrac d{p_0}.
\end{equation}
Thus taking $(a^{-1},b^{-1})$ slightly inside and using real interpolation, 
$$
\|\scriptF_{\tilde\gamma} f\|_{L^q} \leq C_{d,p} \|f\|_{L^p}, \qquad q=\tfrac{d(d+1)}2 p', \quad \tfrac dp > \tfrac 2{d+2} + \tfrac{d-2}{(d+2)p_0}.
$$
This completes the proof of the lemma and thus of Theorem~\ref{T:local rest}.
\end{proof}

\subsection*{Comparison with prior work}
With these arguments in place, it is easier to put our result and our approach in context.  Let $I$ be an interval and $\gamma:I \to \R^d$ be a $C^d$ curve.  To obtain bounds in the Christ range, $q \geq \tfrac{d^2+2d}2$, it is not necessary to deal with the offspring curves; estimate \eqref{E:geometric ineq} with $h=0$ (which we will call the basic geometric inequality), together with some control on the growth of $L_\gamma$, is sufficient (see \cite{ChTAMS, DW, DrMa87}).  For general polynomial curves, the basic geometric inequality is not quite true globally, but in \cite{DW}, Dendrinos--Wright proved a sufficiently uniform substitute (Lemma~\ref{L:polynom decomp}).  

All known proofs beyond the Christ range rely on the method of offspring curves, and this method has seemed much more difficult when $L_\gamma$ is not roughly constant as in \eqref{E:roughly constant}.  Two types of complications arise:  geometric and analytic.  On the geometric front, \eqref{E:bigger torsion} is simply not possible and must be replaced by $|L_{\gamma_h}(t)| \gtrsim |L_\gamma(t)|$, $t \in I_h$.  Even with this adjustment, in practice it has been somewhat easier to prove the basic geometric inequality than the appropriate analogue of Lemma~\ref{L:uniform DM} for those classes of curves (nondegenerate ones and sufficiently small perturbations of monomials) for which both are known.  On the analytic front, in the weighted case, the real interpolation following \eqref{E:real interp} is more difficult (cf.\ \cite{BOS09}).  

Now we turn to the perturbed monomial case.  Let $a_1 < \cdots < a_d$ be real numbers and let $\gamma(t) = (t^{a_1}\theta_1(t),\ldots,t^{a_d}\theta_d(t))$, with $\theta_i \in C^d([0,1])$ and $\theta_i(0) \neq 0$.   Drury--Marshall proved the analogue of Lemma~\ref{L:uniform DM} in the case $\theta_i\equiv 1$ (an omission in the perturbed case is noted in \cite{DeMu}), and used this to obtain restriction estimates off the sharp line (sharp in the Christ range).  Finally, more than twenty years later, Bak--Oberlin--Seeger overcame the analytic difficulties and proved the full range of restriction estimates for these curves \cite{BOS09}, again when $\theta_i \equiv 1$.  In \cite{DeMu}, Dendrinos--M\"uller proved that in the general case, there exists a constant $\delta= \delta_\gamma > 0$ such that the analogue of Lemma~\ref{L:uniform DM} holds on $I=[0,\delta]$ and, arguing similarly to \cite{BOS09}, showed that this gives the full range of $L^p \to L^q$ estimates for restriction to $\gamma|_{[0,\delta]}$.  It is not claimed in \cite{DeMu}, but this implies a non-uniform result for polynomials.  Indeed, if $\gamma$ is a polynomial, then using the reparametrization $t \mapsto t^{-1}$ near $\pm\infty$, near every point of $\R \cup\{\pm\infty\}$, after an affine transformation, $\gamma$ equals a small perturbation of a monomial curve, and so the result follows by compactness of $\R \cup \{\pm\infty\}$.  

Unfortunately, the argument outlined above does not suggest an approach toward proving uniform estimates, even in the polynomial case.  This is not merely a technical issue.  Dating back to Sj\"olin's theorem on convex plane curves \cite{Sjolin}, a major goal has been establishing uniform estimates, and in fact the chief motivation for studying the polynomial case at all has been that it is very clear what the uniform result should be.  Shortly after \cite{DeMu}, Bak--Oberlin--Seeger \cite{BOS11} extended the Dendrinos--Wright range in the general polynomial case by using a result from \cite{BSjfa}, and also proved the full range of estimates for simple curves $\gamma(t) = (t,\ldots,t^{d-1},\phi(t))$.  

There had been some evidence that obtaining the remaining uniform estimates for polynomial curves might be substantially more difficult.  The proof of the basic geometric inequality for polynomials is long and highly nontrivial (it constitutes the bulk of \cite{DW}), and by comparison with the proof of Proposition~8 in \cite{DeMu}, the uniform version seems substantially harder than the local version.  (Roughly, the freedom to choose $\delta_\gamma$ to depend on $\gamma$ seems to make the problem somewhat easier.)  The task of proving a sufficiently uniform global version of Lemma~\ref{L:uniform DM} for polynomials seems potentially even more daunting than the basic geometric inequality, and because the constant $\delta_\gamma$ depends on the local behavior of the curve in a complicated way (and because it is even more difficult to tease out the local behavior after applying the affine transformations needed to make the curve locally monomial-like), the arguments of \cite{DeMu} do not seem to offer a clear path forward.  

We will avoid the above-mentioned geometric and analytic complications by localizing to dyadic torsion scales.  Our task for the next two sections is to show that this localization is reasonable by recovering the global restriction estimate.

\section{A uniform square function estimate} \label{S:square}

The following lemma is essentially due to Dendrinos--Wright in \cite{DW}.  

\begin{lemma} \label{L:polynom decomp}
Let $\gamma:\R \to \R^d$ be a polynomial of degree $N$, and assume that $L_\gamma \not\equiv 0$.  We may decompose $\R$ as a disjoint union of intervals,
$$
\R = \bigcup_{j=1}^{M_{N,d}} I_j,
$$
so that for $t \in I_j$,
\begin{equation} \label{E:L sim}
|L_\gamma(t)| \sim A_j|t-b_j|^{k_j}, \qquad |\gamma_1'(t)| \sim B_j |t-b_j|^{\ell_j},
\end{equation}
and for all $(t_1,\ldots,t_d) \in I_j^d$, 
\begin{equation} \label{E:dw gi}
|J_\gamma(t_1,\ldots,t_d)| \gtrsim \prod_{j=1}^d |L_\gamma(t_j)|^{\frac 1d}\prod_{1 \leq i < j \leq d} |t_j-t_i|,
\end{equation}
where for each $j$, $k_j$ and $\ell_j$ are integers satisfying $0 \leq k_j \leq dN$ and $0 \leq \ell_j \leq N$, and the centers $b_j$ are real numbers not contained in the interior of $I_j$.  Furthermore, the map $(t_1,\ldots,t_d) \mapsto \sum_{j=1}^d \gamma(t_j)$ is one-to-one on $\{t \in I_j^d : t_1 < \cdots < t_d\}$.    The implicit constants and $M_{N,d}$ depend on $N$ and $d$ only.  
\end{lemma}

The main difficulty in proving this lemma is establishing \eqref{E:dw gi}.  Fortunately, this has already been done in \cite{DW}, and we will make no attempt to recap the lengthy argument.  As for the rest, strictly speaking, Dendrinos--Wright prove this lemma without the estimate on $\gamma_1'$, but  Lemma~\ref{L:polynom decomp} may be obtained from their theorem in a straightforward manner.  We briefly explain how this can be done. 

\begin{proof}[The deduction of Lemma~\ref{L:polynom decomp} from \cite{DW}]
 Two decomposition procedures are employed in \cite{DW}.  

For the first, given a polynomial $Q$ and interval $I$, the `D1' procedure decomposes $\R$ into a union of $O(\deg Q)$ intervals, $I = \bigcup J$ such that on $J$, $|Q(t)| \sim A_J |t-b_J|^{k_J}$, where $b_J$ is the real part of a zero of $Q$ and $k_J$ is an integer with $0 \leq k_J \leq \deg Q$.  

The second procedure is due to Carbery--Ricci--Wright in \cite{CRW}.  Given a polynomial $P$ and a center $b$, the `D2' procedure decomposes $\R$ into a union of $O(\deg P)$ `gaps' and $O(\deg P)$ `dyadic intervals.'  On a gap, $|P(t)| \sim A|t-b|^k$, for some integer $0 \leq k \leq \deg P$.  On a dyadic interval $|t-b| \sim A$ for some constant $A$.  

From \cite{DW}, we know that it is possible to decompose $\R = \bigcup I$ in such a way that for each $I$, \eqref{E:dw gi} holds on $I^d$.  Performing the D1 procedure with $Q = L_\gamma$, we may additionally assume that $L_\gamma(t) \sim A_I|t-b_I|^{k_I}$ on $I$.  We fix $I$ and perform the D2 decomposition with $P = \gamma_1'$ and $b=b_I$.  If $G$ is a gap interval, the conclusions of the lemma hold on $G \cap I$.  If $D$ is a dyadic interval, we perform the D1 decomposition with $Q=\gamma_1'$ on $D \cap I$.  If $J\subset G \cap I$ is an interval resulting from this decomposition, $|\gamma_1'(t)| \sim B_J|t-b_J|^{k_J}$ and $|L_\gamma(t)| \sim A_J$ on $J$, so the conclusions of the lemma hold.
\end{proof}

Our square function estimate is the following.

\begin{proposition} \label{P:square}
Let $\gamma:\R \to \R^d$ be a polynomial of degree $N$, and assume that $L_\gamma \not\equiv 0$.  Let $\{I_j\}_{j=1}^{M_{N,d}}$ denote the collection of intervals from Lemma~\ref{L:polynom decomp}.  Fix $j$.  For $n \in \Z$, define
$$
I_{j,n} = \{t \in I_j :  2^n \leq |t-b_j| < 2^{n+1}\}.
$$
Then for each for each $(p,q)$ satisfying $q=\tfrac{d(d+1)}2 p'$ and $\infty > q>\tfrac{d^2+d+2}2$, $f \in L^p$, and $j$,
\begin{equation} \label{E:square}
\|\mathcal E_\gamma (\chi_{I_j}f)\|_{L^q(\R^d)} \lesssim \bigl\| \bigl(\sum_n \bigl| \mathcal E_\gamma (\chi_{I_{j,n}} f)\bigr|^2\bigr)^{\frac12}\bigr\|_{L^q(\R^d)}.
\end{equation}
The implicit constant depends only on $N,d,q$.  
\end{proposition}

\begin{proof}[Proof of Proposition~\ref{P:square}]
If $k_j=0$, the right side of \eqref{E:square} involves $O(1)$ values of $n$, so the inequality is trivial.  We assume henceforth that $k_j > 0$.  

By standard approximation arguments, we may assume that $f$ is supported in the union of finitely many of the $I_{j,n}$.  Thus by Theorem~\ref{T:local rest}, $\mathcal E_\gamma(\chi_{I_j} f) \in L^q$.  

By reparametrizing, we may assume that $b_j=0$ and that $I_j \subset [0,\infty)$.  Let $a_j$ be the left-hand endpoint of $I_j$.  Applying an affine transformation if necessary, we may assume that $A_j = B_j = 1$ and $\gamma(a_j) = 0$.  

Let $n_j = \lceil\log_2 a_j\rceil + 1$.  To prove \eqref{E:square}, it suffices by the triangle inequality to prove that
$$
\|\mathcal E_\gamma (\chi_{I_j \cap [2^{n_j},\infty)}f)\|_{L^q(\R^d)} \lesssim \bigl\| \bigl(\sum_{n \geq n_j} \bigl| \mathcal E_\gamma (\chi_{I_{j,n}} f)\bigr|^2\bigr)^{\frac12}\bigr\|_{L^q(\R^d)}.
$$
(The left side excludes at most two of the $I_{j,n}$.)  By the fundamental theorem of calculus, 
$$
\gamma_1(t) = \int_{a_j}^t \gamma_1'(s)\, ds \sim t^{\ell_j+1} - a_j^{\ell_j+1} \sim t^{\ell_j+1}
$$
on $I_j \cap [2^{n_j},\infty)$.  Thus for $n \geq n_j$,  $\mathcal E_\gamma(\chi_{I_{n,j}}f)$ has its frequency support contained in
$$
\{\xi \in \R^d : 2^{n\ell_j}\sim |\xi_1|\},
$$
and the proposition follows from by Fubini (to separate out the integral in $x_1$) and standard estimates for the Littlewood--Paley square function (see for instance \cite[Ch. VI]{Big Stein}).  
\end{proof}

\section{Almost orthogonality} \label{S:sum}
 
In this section we use a multilinear estimate to show that pieces at different scales interact weakly, and this allows us to sum.  A key step in the multilinear estimate builds on an argument of Christ in \cite{ChTAMS}, though we must interpolate to arrive at the estimate we need.  The procedure for putting the pieces together is inspired by the bilinear approach to Fourier restriction from \cite{TVV}.  Bilinear and multilinear techniques have recently had considerable success in bounding Fourier restriction operators: \cite{BourgainGuth, LeeHamGen, TaoBilin, WolffCone}.  Additional motivation comes from some recent work \cite{DSjfa, DSinprep} on generalized Radon transforms.

\begin{proof}[Proof of Theorem~\ref{T:main}]
By the triangle inequality, it suffices to prove that
$$
\|\mathcal E_\gamma f\|_{L^q(\R^d)} \leq C_{N,d,p}\|f\|_{L^p(\lambda_\gamma)},
$$
when $f$ is supported on one of the intervals $I:= I_j$ from the decomposition in Lemma~\ref{L:polynom decomp}.  By Theorem~\ref{T:local rest}, we already know that this estimate holds if $|L_\gamma| \sim A_j$ on $I_j$, so we may assume that $|L_\gamma(t)| \sim A_j |t-b_j|^{k_j}$ with $k_j \geq 1$.  Making an affine transformation and reparametrization to $\gamma$ if necessary, we may further assume that $A_j = 1$, $b_j = 0$, and $\gamma(0) = 0$.  As mentioned in the introduction, the theorem is known when $q \geq \tfrac{d^2+2d}2$, so we may assume that $q \leq d^2+d$.  These assumptions will remain in place for the remainder of the article.  

To avoid unreadable subscripts, let $D=\tfrac{d^2+d}2$.  

Let $I_n := I_{j,n}$ be the intervals from Proposition~\ref{P:square}.  By Proposition~\ref{P:square}, arithmetic, Minkowski's inequality (since $\frac{q}{2D} \leq 1$), and more arithmetic, we have
\begin{align*}
&\|\mathcal E_\gamma f\|_{L^q}^q \lesssim \int \bigl(\sum_n |\mathcal E_\gamma (\chi_{I_n} f)|^2\bigr)^{\frac q2}\, dx
= \int \prod_{j=1}^D \bigl(\sum_{n_j} |\mathcal E_\gamma (\chi_{I_{n_j}} f)|^2\bigr)^{\frac{q}{2D}}\, dx\\
&\qquad \leq \int\prod_{j=1}^D \sum_{n_j} |\mathcal E_\gamma(\chi_{I_{n_j}} f)|^\frac{q}{D}\, dx
\sim \sum_{n_1 \leq \cdots \leq n_D} \int  \prod_{j=1}^D |\mathcal E_\gamma(\chi_{I_{n_j}} f)|^{\frac{q}{D}}\, dx.
\end{align*}

\begin{lemma}  \label{L:multilinear}  There exists $\eps = \eps_{d,N,p}>0$ such that under the above assumptions on $I$, $\gamma$, and $q$, if $q=D p'$, $n_1 \leq \cdots \leq n_D$, and $f_j$ is an $L^p$ function supported in $I_{n_j}$, $1 \leq j \leq D$, we have
\begin{equation} \label{E:multilinear}
\| \prod_{j=1}^D \mathcal E_\gamma f_j\|_{L^\frac{q}{D}} \leq C_{d,N,p,q} 2^{-\eps(n_D-n_1)}\prod_{j=1}^{D}\|f_j\|_{L^p(\lambda_\gamma)}.
\end{equation}
\end{lemma}

We postpone the proof of the lemma for now to complete the proof of the theorem.

By the lemma and the estimate immediately preceding it,
\begin{equation} \label{E:E to nj}
\|\mathcal E_\gamma f\|_{L^q}^q \lesssim \sum_{n_1 \leq \cdots \leq n_D} 2^{-\eps|n_D-n_1| }\prod_{j=1}^D\|\chi_{I_n}f\|_{L^p(\lambda_\gamma)}^{q/D}.
\end{equation}
Fix $n_1,n_D$.  We recycle notation by defining
$$
I_{n_1,n_D} := [2^{n_1},2^{n_D}].
$$
Of course, if $n_1 \leq n_2 \leq \cdots \leq n_D$, then $I_{n_j} \subset I_{n_1,n_D}$; moreover, there are only $O((n_D -n_1)^D)$ choices for $n_2,\ldots,n_{D-1}$.  Combining this with \eqref{E:E to nj},
\begin{equation} \label{E:E to Inm}
\|\mathcal E_\gamma f\|_{L^q}^q \lesssim \sum_m \sum_n 2^{-\eps m} m^{C_d}\|\chi_{I_{n,n+m}} f\|_{L^p(\lambda_\gamma)}^q.
\end{equation}

Fix $m$.  By H\"older's inequality and the fact that each point is contained in at most $m$ intervals of the form $I_{n,n+m}$,
\begin{align*}
\sum_n \|\chi_{I_{n,n+m}} f\|_{L^p(\lambda_\gamma)}^q &\leq (\sup_n \|\chi_{I_{n,n+m}} f\|_{L^p(\lambda_\gamma)}^{q-p}) \sum_n \|\chi_{I_{n,n+m}} f\|_{L^p(\lambda_\gamma)}^p \\
&\lesssim m \|f\|_{L^p(\lambda_\gamma)}^q.
\end{align*}
Combining this with \eqref{E:E to Inm},
$$
\|\mathcal E_\gamma f\|_{L^q}^q \lesssim \sum_m 2^{-\eps m} m^{C_d+1} \|f\|_{L^p(\lambda_\gamma)}^q \lesssim \|f\|_{L^p(\lambda_\gamma)}^q.
$$
Thus the only thing left to establish Theorem~\ref{T:main} is the proof of Lemma~\ref{L:multilinear}.
\end{proof}  

\begin{proof}[Proof of Lemma~\ref{L:multilinear}]
By H\"older's inequality,
\begin{align*}
\|\prod_{j=1}^D \mathcal E_\gamma f_j\|_{L^{\frac{q}{D}}}
&\leq \prod_{i=d+1}^{D}\|\mathcal E_\gamma f_{j_i}\|_{L^q}\|\prod_{i=1}^d \mathcal E_\gamma f_{j_i}\|_{L^{\frac qd}},
\end{align*}
whenever the $j_i$ are any enumeration of $\{1,\ldots,D\}$.  Thus it suffices to prove that
\begin{equation} \label{E:d linear}
\|\prod_{j=1}^d \mathcal E_\gamma  f_j\|_{L^{\frac qd}} \leq C_{N,d} 2^{-\eps(n_d-n_1)}\prod_{j=1}^d \|f_j\|_{L^p(\lambda_\gamma)}, \:\text{whenever}\: n_1 \leq \cdots \leq n_d.
\end{equation}

Next, by H\"older's inequality and Theorem~\ref{T:local rest} (by our assumption on the supports of the $f_j$), 
\begin{align*}
\|\prod_{j=1}^{d} \mathcal E_\gamma f_j\|_{L^{\frac{q}{d}}}
&\leq \prod_{j=1}^{d} \|\mathcal E_\gamma f_j\|_{L^q}
\leq C_{d,N,p} \prod_{j=1}^{d} \|f_j\|_{L^p(\lambda_\gamma)}.
\end{align*}
Thus it suffices to prove \eqref{E:d linear} when $n_d \geq n_1 + 2d$.  Furthermore, by complex interpolation (with some $(p,q)$ sufficiently near the endpoint), it suffices to prove \eqref{E:d linear} when $q=d(d+1)$ and $p=2$.  In other words, we have reduced matters to proving 
\begin{equation} \label{E:rst}
\|\prod_{j=1}^d \mathcal E_\gamma f_j\|_{L^{d+1}} \lesssim 2^{-\eps(n_d-n_1)}\prod_{j=1}^d \|f_j\|_{L^2(\lambda_\gamma)}.
\end{equation}

By Hausdorff--Young,
\begin{equation} \label{E:HY}
\|\prod_{j=1}^d \mathcal E_\gamma  f_j\|_{L^{d+1}} \leq \|  (d\mu_1)*\cdots *(d\mu_d)\|_{L^{\frac{d+1}d}},
\end{equation}
where $d\mu_j$ is the measure defined by
$$
d\mu_j(\phi) = \int_{I_{n_j}} \phi(\gamma(t)) f_j(t)\lambda_\gamma(t)\, dt.
$$
We compute
\begin{align*}
[(d\mu_1)*\cdots *(d\mu_d)](\phi) &= \int \phi(\sum_{i=1}^d \gamma(t_i))\prod_{i=1}^d f_i(t_i) \lambda_\gamma(t_i)\, dt\\
&= \sum_{\sigma \in S_d} \int_{P_\sigma} \phi(\sum_{i=1}^d \gamma(t_i))\prod_{i=1}^d f_i(t_i) \lambda_\gamma(t_i)\, dt,
\end{align*}
where $S_d$ is the symmetric group on $d$ letters, and for each $\sigma \in S_d$, 
$$
P_\sigma = \{(t_1,\ldots,t_d) \in I_{n_1} \times \cdots\times I_{n_d} : t_{\sigma(1)} < \cdots < t_{\sigma(d)}\}.
$$

By Lemma~\ref{L:polynom decomp}, $\Phi(t) := \sum_{j=1}^d \gamma(t_j)$ is one-to-one on $P_\sigma$ , so we make the change of variables $\xi=\Phi(t)$, yielding
$$
[(d\mu_1)*\cdots *(d\mu_d)]= \sum_{\sigma \in S_d} F_\sigma,
$$
where 
$$
F_\sigma(\xi) =  \chi_{P_\sigma}(t)(\prod_{i=1}^d f_i(t_i)\lambda_\gamma(t_i)) |J_\gamma(t_1,\ldots,t_d)|^{-1}|_{t = (\Phi|_{P_\sigma})^{-1}(\xi)}.
$$
By the change of variables formula and the geometric inequality \eqref{E:dw gi},
\begin{align*}
\|F_\sigma\|_{L^\frac{d+1}d} &= 
\| \chi_{P_\sigma}(t) \prod_{i=1}^d f_i(t_i)\lambda_\gamma(t_i) |J_\gamma(t_1,\ldots,t_d)|^{-\frac1{d+1}}\|_{L^{\frac{d+1}d}}\\
&\lesssim \|\prod_{i=1}^d f_i(t_i)\lambda_\gamma(t_i)^{\frac12} \prod_{i < j}|t_i-t_j|^{-\frac1{d+1}}\|_{L^{\frac{d+1}d}}.
\end{align*}

By the pigeonhole principle, there exists an index $k$, $1 \leq k < d$ such that $n_{k+1}-n_k \geq \tfrac{n_d-n_1}d$.  In particular, $n_{k+1}-n_k \geq 2$.  Therefore for $(t_1,\ldots,t_d) \in I_{n_1} \times \cdots \times I_{n_d}$, 
\begin{align*}
\prod_{1 \leq i < j \leq d} |t_i-t_j| &= \prod_{i\leq k, j \geq k+1 }|t_i-t_j| \prod_{1 \leq i < j \leq k} |t_i-t_j| \prod_{k+1 \leq i < j \leq d} |t_i-t_j|\\
&\sim \prod_{i \leq k, j \geq k+1} 2^{n_j} \prod_{1 \leq i < j \leq k} |t_i-t_j| \prod_{k+1 \leq i < j \leq d} |t_i-t_j|.
\end{align*}
This implies that
\begin{equation} \label{E:F less T}
\|F_\sigma\|_{L^{\frac{d+1}d}}^{\frac{d+1}d} \sim 2^{-\frac kd (n_{k+1}+\cdots + n_d)} T_k(f_1,\ldots,f_k) \times T_{d-k}(f_{k+1},\ldots,f_d),
\end{equation}
where 
$$
T_\ell(g_1,\ldots,g_\ell) = \int_{\R^\ell} \prod_{i=1}^\ell g_i(t_i)^{\frac{d+1}d}\lambda_\gamma(t_i)^{\frac{d+1}{2d}} \prod_{1 \leq i < j \leq \ell} |t_i-t_j|^{-\tfrac1d} \, dt_1\,\cdots\,dt_\ell.
$$

In the proof of Proposition~2.2 of \cite{ChTAMS}, it is shown (see inequality (2.1) of that article) that if $1 \leq p < \ell$ and $p^{-1}+\tfrac{\ell-1}2 q^{-1} = 1$, then 
$$
\int_{\R^\ell} \prod_{i=1}^\ell f_i(t_i) \prod_{1 \leq i < j \leq \ell} g_{ij}(t_i-t_j)\, dt \lesssim \prod_{i=1}^\ell \|f_i\|_{L^p} \prod_{1 \leq i < j \leq \ell} \|g_{ij}\|_{L^{q,\infty}}.
$$
We apply this to $T_k$ with $q=d$ and $p = \tfrac{2d}{2d-k+1}$ (if $k=1$, $p=k$, but then the inequality is trivial) to see that
$$
|T_k(f_1,\ldots,f_k)| \lesssim \prod_{i=1}^k \|f_i \lambda_\gamma^{\frac 12}\|_{L^{\frac{2(d+1)}{2d-k+1}}}^{\frac{d+1}d}.
$$
Since $\tfrac{2(d+1)}{2d-k+1} < 2$, by H\"older's inequality and the fact that $|\supp f_j| \leq |I_{n_j}| \sim 2^{n_j}$, this implies that
$$
|T_k(f_1,\ldots,f_k)| \lesssim \prod_{i=1}^k 2^{\frac{n_i(d-k)}{2d}} \|f_i\|_{L^2(\lambda_\gamma)}^{\frac{d+1}d}.
$$
Similarly, 
$$
|T_{d-k}(f_1,\ldots,f_k)| \lesssim \prod_{i=k+1}^d 2^{\frac{n_i k}{2d}} \|f_i\|_{L^2(\lambda_\gamma)}.
$$
Inserting these estimates into \eqref{E:F less T} and performing a bit of arithmetic, 
\begin{align*}
\|F_\sigma\|_{L^{\frac{d+1}d}}^{\frac{d+1}d} &\lesssim 2^{\frac{(n_1+\cdots+n_k)(d-k)}{2d} - \frac{(n_{k+1}+\cdots+n_d)k}{2d}} \prod_{i=1}^d \|f_i\|_{L^2(\lambda_\gamma)}\\
&\lesssim 2^{-\frac{k(d-k)}{2d}(n_k-n_{k+1})}\prod_{i=1}^d \|f_i\|_{L^2(\lambda_\gamma)}.
\end{align*}
Since $n_k-n_{k+1} \geq \tfrac1d(n_d-n_1)$ by our choice of $k$, this completes the proof of \eqref{E:d linear} and hence of Lemma~\ref{L:multilinear}.
\end{proof}


\section{Proof of the corollary}\label{S:interp}

We begin with \eqref{E:unweighted}.  Let $I_{\rm{lo}} = \{|L_\gamma| < 1\}$ and $I_{\rm{hi}} = \R \setminus I_{\rm{lo}}$.  With $Z_\gamma$ equal to the set of complex zeroes, we estimate
\begin{equation} \label{E:I lo hi}
1 \gtrsim_\gamma |L_\gamma| \gtrsim_\gamma d_\gamma^{K_{\rm{min}}}, \: \text{on}\: I_{\rm{lo}}, \quad |L_\gamma| \sim_{\gamma} d_\gamma^{K_{\rm{max}}} \gtrsim 1, \:\text{on}\:I_{\rm{hi}}.
\end{equation}
If $p \leq q$, the full range of estimates follow by writing $\hat f \circ \gamma = (\hat f \circ \gamma)\chi_{I_{\rm{lo}}}+(\hat f \circ \gamma)\chi_{I_{\rm{hi}}}$ and applying \eqref{E:I lo hi}, \eqref{E:DM interp}, and the embedding $L^q \subseteq L^{q,p}$.  

Now assume that $p < q$ and $N_{\rm{min}}q < p' < N_{\rm{max}}q$.  Let $I_n = \{2^n \leq |L_\gamma| < 2^{n+1}\}$.  Then 
\begin{equation} \label{E:In}
|I_n| \sim_\gamma 2^{n/K_{\rm{min}}},\:\: n \leq 0, \qquad  |I_n| \sim_\gamma 2^{n/K_{\rm{max}}} ,\:\:  n \geq 0.
\end{equation}
(It is crucial that $\gamma$ is a polynomial.)  Let $\tilde q = \tfrac 2{d^2+d}p'$; then $q < \tilde q$.  By H\"older's inequality, \eqref{E:In}, and our main theorem, if $n \geq 0$,
\begin{align*}
\|\hat f \circ \gamma\|_{L^q(I_n)} &\lesssim  2^{n(\frac1{qK_{\rm{max}}}-\frac1{\tilde q K_{\rm{max}}}-\frac{2}{(d^2+d)\tilde q})} \|\hat f \circ \gamma\|_{L^p(\R^d)}.
\end{align*}
The exponent is just $\tfrac1{K_{\rm{max}}}(\tfrac1q - \tfrac{N_{\rm{max}}}p)<0$, so this is summable over $n \geq 0$.  The argument when $n \leq 0$ is similar.  

Now we turn to the optimality of \eqref{E:unweighted}.  The condition $p < p_d = \tfrac{d^2+d+2}{d^2+d}$ for $q \geq 1$ follows by considering a nondegenerate segment along $\gamma$ and applying a result of Arkhipov--Chubarikov--Kuratsuba in \cite{AKC} (see also \cite{many dudes}).  That $p < p_d$ for $q<1$ follows by interpolation with the restricted strong type $L^{p_d,1} \to L^{p_d}$ estimate for nondegenerate curves from \cite{BOS09}.  It suffices by interpolation to verify the remaining inequalities when $q \geq 1$.  Thus we may consider $L^{q'} \to L^{p'}$ bounds for the extension operator.  The condition $N_{\rm{min}}q \leq p' \leq N_{\rm{max}}q$ follows from Knapp type examples.  Assume $p > q$.  Performing an affine transformation and reparametrization, we may assume that $\gamma(t) = (t^{m_1}\theta_1(t),\ldots,t^{m_d}\theta_d(t))$, where $m_1+\cdots+m_d = N_{\rm{min}}$ and the $\theta_i$ are polynomials with $\theta_i(0)=1$ and $\|\theta_i'\|_{C^0}$ sufficiently small on $[0,1]$.  Let $N$ be a large integer and define
\begin{align*}
g_n(t) &=  2^{\frac{n}{q'}}e^{ix_n\gamma(t)} \chi_{[2^{-n},2^{-n+1})}, \: 1 \leq n \leq N; \quad g = \sum_1^N g_n,
\end{align*}
where the $x_n$ are spaced sufficiently far apart in $\R^d$.  Using Knapp-type arguments and the physical space separation of the $\scriptF_\gamma g_n$, if $p'=N_{\rm{min}}q$,
$$
\|g\|_{L^{q'}} \sim N^{\frac1{q'}}, \qquad \|\mathcal F_\gamma g\|_{L^{p'}} \sim N^{\frac1{p'}}.
$$
Letting $N \to \infty$, since $q' > p'$, $L^{q'} \to L^{p'}$ boundedness cannot hold.  The verification of $p' < N_{\rm{max}}q$ is similar.  

This leaves us to prove \eqref{E:DM interp}.  We begin with a well-known  lemma.

\begin{lemma}  Let $P$ be a real polynomial and $Z_\gamma$ be the set of complex zeroes of $P$.  There exists a decomposition $\R = \bigcup_{j=1}^{C(\#Z_\gamma, \deg P)}I_j$ as a union of intervals such that the  following holds for each $j$.  There exist $C_j>0$, $b_j \in Z_\gamma$, and $k_j \in \Z_{\geq 0}$ so that for every $t \in I_j$, $b_j$ is the closest element of $Z_\gamma$ to $t$ and $|P(t)| \sim C_j |t-b_j|^{k_j}$.  The implicit constant depends only on $\deg P$.  
\end{lemma}  

This is a simple variant of a well known (see e.g.\ \cite{CRW}) lemma; for the convenience of the reader, we give the short proof. 

\begin{proof} Define $I_b = \{t \in \R: d_\gamma = |t-b|\}$, $b \in Z_\gamma$.  Of course, $\R = \bigcup_b I_b$, so we may leave $b$ fixed for the remainder of the argument.  We index the elements of $Z_\gamma$ by $b=b_0,\ldots,b_M$, so that $|b-b_0| \leq \cdots \leq |b-b_M|$.  For convenience, we also set $b_{M+1} = \infty$.  Define a sequence of annuli, 
\begin{gather*}
A_j = \{t \in I_b : \tfrac12|b-b_j| \leq |t-b| \leq \tfrac12|b-b_{j+1}|\}, \qquad 0 \leq j \leq M.
\end{gather*}
Then $I_b = \bigcup A_j$, and on $A_j$, by the triangle inequality and the definition of $I_b$, 
$$
|P(t)| = C \prod_{k=0}^M|t-b_k|^{n_k} \sim C \prod_{k=j+1}^M |b-b_k|^{n_k}|t-b|^{n_0+\cdots+n_j},
$$
where we are taking $n_k = 0$ if $P(b_k) \neq 0$.  
\end{proof}

Now we return to the proof of \eqref{E:DM interp}.  We adapt an argument of Drury--Marshall \cite{DrMa85} to the polynomial case.  By the triangle inequality, it suffices to prove the estimate when the $L^{q,r}_t$ norm is restricted to one of the intervals $I=I_j$ from our lemma.  Let $C,k,b$ denote the corresponding constant, power, and zero.  In particular, $d_\gamma(t) = |t-b|$ on $I$.  Applying an affine transformation to $\gamma$ if necessary (this leaves \eqref{E:DM interp} invariant), we may assume that $C=1$.  

We rewrite our main result \eqref{E:restriction pq} as 
$$
\|\hat f(\gamma(t))|t-b|^{\frac k{p'}}\|_{L^q(\lambda_\gamma(t)\, dt)} \leq C_{p,d,N} \|f\|_{L^p_x}, \qquad p' = \tfrac{d(d+1)}2 q, \quad 1 \leq p < \tfrac{d^2+d+2}{d^2+d}.
$$
For any $r > 0$, $|t-b|^{-\frac1r}$ is in $L^{r,\infty}$, so by the Lorentz space version of H\"older's inequality \cite{SteinWeiss}, 
$$
\|\hat f(\gamma(t)) |t-b|^{\frac k{p'}-\frac1c+\frac1q}\|_{L^{c,q}_t} \leq C_{p,c,d,N}\|f\|_{L^p_x}, 
$$
whenever $p'=\tfrac{d(d+1)}2 q$, $c \leq q$, $1 \leq p < \tfrac{d^2+d+2}{d^2+d}$.   

Marcinkiewicz interpolation \cite{SteinWeiss} along the level sets of $\frac k{p'} - \frac1c+\frac1q$, within the region $p'=\tfrac{d(d+1)}2 q$, $c \leq q$, $1 < p < \tfrac{d^2+d+2}{d^2+d}$, gives the estimate 
$$
\|\hat f(\gamma(t)) |L_\gamma(t)|^{\frac1{p'}} d_\gamma(t)^{-\frac1q+\tfrac{d(d+1)}{2p'}}\|_{L^{q,p}_t(I)} \leq C_{p,q,N,k}\|f\|_{L^p_x}, 
$$ 
and since $0 \leq k \leq N$, we can put the pieces back together to obtain \eqref{E:DM interp}.



\end{document}